\documentclass[11pt]{article}
\usepackage{graphicx}
\usepackage{amssymb, amsthm, amsmath, amsfonts}
\usepackage{mathrsfs,microtype}
\usepackage{epstopdf, color}
\usepackage{tikz}
\usepackage{hyperref}

\usetikzlibrary{arrows,decorations.pathmorphing,decorations.footprints,fadings,calc,trees,mindmap,shadows,decorations.text,patterns,positioning,shapes,matrix,fit,through,decorations.pathreplacing}
\usepackage[margin=1in]{geometry}

\theoremstyle{plain}
\newtheorem{theorem}{Theorem}[section]

\theoremstyle{definition}
\newtheorem*{definition}{Definition}

\theoremstyle{remark}

\newcommand{\calG}{\mathcal{G}}
\newcommand{\calT}{\mathcal{T}}

\newcommand{\xhom}{\text{xhom}}

\title{Maximizing and minimizing the number of generalized colorings of trees}
\author{John Engbers\thanks{john.engbers@marquette.edu; Department of Mathematics, Statistics and Computer Science, Marquette University, Milwaukee, WI 53201. Research supported by the Simons Foundation grant 524418.}  \and Christopher Stocker\thanks{christopher.stocker@marquette.edu; Department of Mathematics, Statistics and Computer Science, Marquette University, Milwaukee, WI 53201.} }
\date{\today}

\begin{document}

\maketitle

\begin{abstract}
We classify the trees on $n$ vertices with the maximum and the minimum number of certain generalized colorings, including conflict-free, odd, non-monochromatic, star, and star rainbow vertex colorings. We also extend a result of Cutler and Radcliffe on the maximum and minimum number of existence homomorphisms from a tree to a completely looped graph on $q$ vertices.  
\end{abstract}



\section{Introduction}

Let $G = (V(G),E(G))$ be a simple graph, and let $c:V(G) \to \{1,2,3\ldots\}$ be a coloring of the vertices of $G$. A \emph{proper coloring} of $G$ is a coloring so that no edge of $G$ is monochromatic. When $q$ colors are used (i.e. $c:V(G) \to \{1,\ldots,q\}$) we will often refer to a proper coloring as a \emph{proper $q$-coloring}. 

Other types of vertex colorings have recently been investigated.  One such variation places the allowable colors as the vertices in a graph $H$ and joins two vertices of $H$ with an edge if those colors can appear across an edge in $G$.  For a given $H$, an \emph{$H$-coloring of $G$}, or \emph{graph homomorphism from $G$ to $H$}, is a coloring of $G$ using the scheme from the graph $H$; more precisely, an $H$-coloring of $G$ is a function $f: V(G) \to V(H)$ so that if $v,w \in V(G)$ with $vw \in E(G)$, then $f(v)f(w) \in E(H)$. Notice that when $H=K_q$, the complete graph on $q$ vertices, an $H$-coloring of $G$ corresponds to a proper $q$-coloring of $G$. When $H$ is an edge with one looped endvertex, an $H$-coloring of $G$ corresponds to an \emph{independent set}, or \emph{stable set}, in $G$.

Finding an $H$-coloring of a graph $G$ can be difficult, and so much recent research has investigated a related extremal problem: \emph{Given a family of graphs $\calG$, which $G \in \calG$ has the largest (smallest, respectively) number of $H$-colorings?} An answer to this question produces bounds on the number of $H$-colorings for any graph in $\calG$, and also implies bounds on the probability that a random coloring of the vertices of $G \in \calG$ from the vertices of $H$ will be an $H$-coloring of $G$. 
Some families $\calG$ that have been considered include regular graphs, graphs with fixed minimum degree, and graphs with a fixed number of edges. For results and conjectures on the extremal $H$-coloring question in these families, we refer the reader to two surveys, \cite{Cutler} and \cite{Zhao}, and the numerous references therein.  

One specific family that will be applicable in this paper is the family of all $n$-vertex trees, which will be denoted by $\calT(n)$.  Extremal independent set counts in trees were first studied in \cite{ProdingerTichy}, while extremal $H$-coloring counts in trees for all other $H$ have also been considered \cite{CsikvariLin,EngbersGalvin,Sidorenko}. In particular, for any $H$ the star is always the tree with the largest number of $H$-colorings, but interestingly the path is not always the tree with the smallest number of $H$-colorings.  See \cite{CsikvariLin} for more details, and \cite{EngbersGalvin} for a class of $H$ such that the path is the tree with the smallest number of $H$-colorings. Note that every tree on $n$ vertices has the same number of proper $q$-colorings. 

We may view $H$-colorings as placing a restriction on the colors that can appear across an edge in $G$. Other vertex coloring schemes have also recently been considered, arising from generalizing proper $q$-colorings by using color restrictions on other subsets of $V(G)$ (and not exclusively on pairs of vertices in $E(G)$). Two natural subsets of vertices to use for color restrictions are $N(v)$ and $N[v]$, which are the neighborhood and closed neighborhood, respectively, of a vertex $v \in V(G)$. The extremal question then becomes the following: \emph{Given a family of graphs $\calG$ and a set of color restrictions, which $G \in \calG$ has the largest (smallest, respectively) number of colorings?} Extremal results for colorings of regular graphs with restricted lists on $N(v)$ and $N[v]$ appear in \cite{CutlerRadcliffe}. 

In this paper, we will focus on this extremal question for a number of types of colorings in the family of $n$-vertex trees $\calT(n)$. 
Let $P_n$ and $S_n$ in $\calT(n)$ denote the path and star on $n$ vertices, respectively. 

\section{Definitions and Statements of Results}\label{sec-thms}

In this section we define various types of colorings and state the corresponding extremal results for those colorings, and in Section \ref{sec-proofs} we provide the proofs of these extremal results. We start with the notion of a conflict-free coloring.

\begin{definition}
A \emph{conflict-free coloring} of a graph $G$ with $q$ colors is a function $c:V(G) \to \{1,\ldots,q\}$ such that for every $v \in V(G)$ there is a color occurring exactly once in $N[v]$. The number of conflict-free colorings of $G$ with $q$ colors is denoted $\chi_{cf}(G;q)$.
\end{definition}

This definition of conflict-free colorings of a graph is a special case of conflict-free colorings of a hypergraph, which was introduced in \cite{NP,SmorThesis}.  Considering conflict-free colorings with restrictions on paths instead of closed neighborhoods was studied in \cite{CheilarisToth}. Finding the minimum number of colors needed to admit a conflict-free coloring, even for disks in the plane, is NP-complete \cite{NP}.   

One application of conflict-free colorings occurs in frequency assignments for cellular networks. Here, the vertices represent base stations and the colors represent frequencies assigned to the base stations. For a client to receive a signal from a base station, they must tune to the signal from some (nearby) base station, and they require that signal to come from only one station in order to avoid signal interference. By representing spatially close base stations via edge adjacency, this frequency assignment problem is modeled by a conflict-free coloring of the associated graph.  For a survey of conflict-free colorings and its applications, see \cite{Smor}.

We determine the trees with the extremal number of conflict-free colorings.

\begin{theorem}\label{thm-conffree}
Let $q \geq 2$ and $T \in \calT(n)$. Then
\[
\chi_{cf}(S_n;q) \leq \chi_{cf}(T;q) \leq \chi_{cf}(P_n;q).
\]
Equality occurs in the upper bound if and only if $T=P_n$.

When $q>2$, equality occurs in the lower bound if and only if $T=S_n$. When $q=2$, equality occurs in the lower bound if and only if $T$ does not contain a non-trivial subtree $T'$ with the property that each vertex in $T'$ has exactly one neighbor outside of $T'$. 
\end{theorem}

We next define two colorings that are related to conflict-free colorings.
\begin{definition}
An \emph{odd coloring} of a graph $G$ with $q$ colors is a function $c:V(G) \to \{1,\ldots,q\}$ such that for each $v \in V(G)$ there is a color occurring an odd number of times in $N[v]$. The number of odd colorings of $G$ with $q$ colors is denoted $\chi_{odd}(G;q)$.
\end{definition}

\begin{definition}
A \emph{star rainbow coloring} of a graph $G$ with $q$ colors is a function $c:V(G) \to \{1,\ldots,q\}$ such that for each $v \in V(G)$ no color occurs more than once in $N[v]$. The number of star rainbow colorings of $G$ with $q$ colors is denoted $\chi_{sr}(G;q)$.
\end{definition}

Odd colorings were first introduced in \cite{CKP}. It is clear that every star rainbow coloring of $G$ is a conflict-free coloring of $G$, and every conflict-free coloring of $G$ is an odd coloring of $G$. We determine the trees with the extremal number of odd and star rainbow colorings. 

\begin{theorem}\label{thm-odd}
Let $q \geq 2$ and $T \in \mathcal{T}(n)$. Then 
\[
\chi_{odd}(S_n;q) \leq \chi_{odd}(T;q) \leq \chi_{odd}(P_n;q).
\]
Equality occurs in the upper bound if and only if $T=P_n$.

When $q>2$, equality occurs in the lower bound if and only if $T=S_n$. When $q=2$, equality occurs in the lower bound if and only if $T$ contains at most one vertex of even degree. 
\end{theorem}

\begin{theorem}\label{thm-rainbowcols}
Let $q \geq 2$ and $T \in \calT(n)$.  Then 
\[
\chi_{sr}(S_n;q) \leq \chi_{sr}(T;q) \leq \chi_{sr}(P_n;q).
\]
Equality occurs in the lower bound if and only if $T=S_n$ and in the upper bound if and only if $T=P_n$.
\end{theorem}

Next we define a non-monochromatic coloring, and give the corresponding extremal result for trees.

\begin{definition}
A \emph{non-monochromatic coloring} of a graph $G$ with $q$ colors is a function $c:V(G) \to \{1,\ldots,q\}$ such that for each $v \in V(G)$ there are at least two colors occurring in $N[v]$. The number of non-monochromatic colorings of $G$ with $q$ colors is denoted $\chi_{nm}(G;q)$.
\end{definition}

\begin{theorem}\label{thm-nm}
Let $q \geq 2$ and $T \in \calT(n)$.  Then 
\[
\chi_{nm}(S_n;q) \leq \chi_{nm}(T;q) \leq \chi_{nm}(P_n;q).
\]
Equality occurs in the lower bound if and only if $T=S_n$ and in the upper bound if and only if $T=P_n$.
\end{theorem}

As these results show, the extremal trees are often $P_n$ or $S_n$.  It is tempting to conjecture that any color restriction on $N[v]$ will cause $P_n$ to maximize the number of colorings among all trees; in other words, that a maximizing graph is independent of the color restriction on $N[v]$. The corresponding statement in the family of regular graphs is in fact true: the regular graph that maximizes the number of colorings given by a restriction on the colors on $N[v]$ or on $N(v)$ is independent of the coloring lists \cite[Theorem 9]{CutlerRadcliffe}. 
The situation for trees, however, is different. We next present a related coloring scheme where the maximizing graph for the number of these colorings is not $P_n$ or $S_n$.

\begin{definition}
A \emph{2-strong-conflict-free coloring} of a graph $G$ with $q$ colors is a function $c:V(G) \to \{1,\ldots,q\}$ such that for each $v \in V(G)$ there are at least two colors that occur exactly once in $N[v]$. The number of 2-strong-conflict-free colorings of $G$ with $q$ colors is denoted $\chi_{2scf}(G;q)$.
\end{definition}

The $2$-strong-conflict-free colorings (in fact, $k$-strong-conflict-free colorings) were originally studied in \cite{Abel} under the name $k$-conflict-free colorings; see also \cite{Smor}.  
For $2$-strong-conflict-free colorings we have the following lower bound.

\begin{theorem}\label{thm-2scf}
Let $q \geq 3$ and $T \in \calT(n)$.  Then 
\[
\chi_{2scf}(P_n;q) \leq \chi_{2scf}(T;q).
\]
Equality occurs if and only if $T=P_n$.
\end{theorem}

Furthermore, for 2-strong-conflict-free colorings neither the path nor the star is the maximizer. Take $n=6$ and $q=3$.  Then it is easy to see that $\chi_{2scf}(P_6;3) = 6$ and $\chi_{2scf}(S_6;3) = 3\cdot 5 \cdot 2 = 30$ (all colorings of $S_6$ are obtained by coloring the center of $S_6$ with one color, exactly one leaf with a second color, and the remaining leaves with the third color).  But now consider the tree $T$ which is a balanced double star, i.e., it has two adjacent vertices $v_1$ and $v_2$ (the centers of the double star), each with two leaves.  There are $6$ ways to have distinct colors on $v_1$ and $v_2$, and $3$ ways to color each pair of leaves, noting that a leaf must have a different color from its neighbor.  There are $3$ ways to have the same color on $v_1$ and $v_2$, and $2$ ways to color each pair of leaves in this situation.  Therefore $\chi_{2scf}(T;3) = 6 \cdot 3 \cdot 3 + 3 \cdot 2 \cdot 2 = 66$. For general $n$, it is not clear what the maximizing tree is in this case. 

These results show that the extremal graphs in $\calT(n)$ are not independent of the list restrictions, unlike the results for regular graphs.  Note that in regular graphs all neighborhoods (closed neighborhoods, respectively) have the same size, whereas the size of a neighborhood (closed neighborhood, respectively) in a tree can vary significantly from vertex to vertex. 

We also consider colorings where the color classes induce a forest of stars. Here the color restrictions are on paths in the tree.

\begin{definition}
A \emph{star coloring} of a graph $G$ with $q$ colors is a function $c:V(G) \to \{1,\ldots,q\}$ such that (1) $v_1v_2 \in E(G)$ implies $c(v_1)\neq c(v_2)$, and (2) for each (not necessarily induced) $P_4$ in $G$, $c|_{P_4}$ maps onto at least three colors. The number of star colorings of $G$ with $q$ colors is denoted $\chi_{s}(G;q)$.
\end{definition}

Star colorings were first introduced by Gr\"unbaum in \cite{Grunbaum}, and \cite{FRR} contains results on the star chromatic number of various families of graphs. We determine the trees with the extremal number of star colorings.

\begin{theorem}\label{thm-starcols}
Let $q \geq 3$ and $T \in \calT(n)$. Then
\[
\chi_s(P_n;q) \leq \chi_s(T;q) \leq \chi_s(S_n;q).
\]
Equality occurs in the lower bound if and only if $T=P_n$ and in the upper bound if and only if $T=S_n$.
\end{theorem}
We remark that the path \emph{minimizes} and the star \emph{maximizes} the number of star colorings in $\calT(n)$, which differs from the corresponding extremal results for conflict-free, odd, star rainbow, and non-monochromatic colorings.

Our techniques also allow us to extend a result from \cite{CutlerRadcliffe} for $q=2$ colors to arbitrary $q$. In that paper, the concept of an existence homomorphism is introduced and investigated.

\begin{definition}[\cite{CutlerRadcliffe}]
Suppose that $G$ and $H$ are graphs with $H$ possibly having loops.  We say that a map $\phi:V(G) \to V(H)$ is an \emph{existence homomorphism} if, for every $v \in V(G)$, there exists a $w \in N(v)$ such that $\phi(v)\phi(w) \in E(H)$. We let $\text{xhom}(G,H)$ be the number of existence homomorphisms from $G$ to $H$.
\end{definition}
In \cite{CutlerRadcliffe} the authors consider $H= E_q^\circ$, the completely looped graph on $q$ isolated vertices, and in this case an existence homomorphism from a graph $G$ to $E_q^\circ$ is a coloring of the vertices of $G$ with $q$ colors so that each color class has no isolated vertices (note that a color class may be empty). They show the following.

\begin{theorem}[Cutler-Radcliffe \cite{CutlerRadcliffe}]\label{thm-CR}
If $T$ is a tree on $n$ vertices, then 
\[
\xhom(T,E_2^\circ) \leq \xhom(P_n,E_2^\circ),
\]
with equality if and only if $T = P_n$.
\end{theorem}

We generalize Theorem \ref{thm-CR} to $q>2$ colors and also find the minimizing tree.

\begin{theorem}\label{thm-xhom}
Let $q \geq 2$ and $T \in \calT(n)$. Then 
\[
\xhom(S_n,E_q^\circ) \leq \xhom(T,E_q^\circ) \leq \xhom(P_n,E_q^\circ).
\]
Equality occurs in the lower bound if an only if $T=S_n$ and in the upper bound if and only if $T=P_n$. 
\end{theorem}

It would be interesting to investigate the maximum and minimum number of these colorings for various other families of graphs, including regular graphs (see \cite[Proposition 18]{CutlerRadcliffe} for maximizing $\xhom(G,E_2^\circ)$ over all $2$-regular graphs), graphs with a fixed minimum degree, and graphs with a fixed number of edges.

\section{Proofs}\label{sec-proofs}
In this section, we present the proofs of the theorems stated in Section \ref{sec-thms}.

\subsection{Odd colorings --- Proof of Theorem \ref{thm-odd}} 

First, we prove the extremal result for odd colorings. Recall that an odd coloring is a vertex coloring where for every vertex $v$ there is a color occurring an odd number of times in $N[v]$.

\begin{proof}[Proof of Theorem \ref{thm-odd}]
Let $T \in \calT(n)$. Notice that if $v \in T$ is a vertex such that $d(v)$ is even, then by parity considerations some color will appear an odd number of times on $N[v]$.  Also note that if $v$ is an uncolored leaf whose neighbor is colored, then there are at most $q-1$ ways to extend the odd coloring to $v$ (by considering only the restrictions on the leaf $v$).  So by coloring all non-leaves first, we see that if $T$ has $k$ leaves then $\chi_{odd}(T;q) \leq q^{n-k}(q-1)^k$.  Since $P_n$ is the unique tree with at most two leaves and in $P_n$ every non-leaf has even degree, the above observations show that $P_n$ is the unique tree maximizing the number of odd colorings, and  $\chi_{odd}(P_n;q) = q^{n-2}(q-1)^2$.

Since every tree $T$ has $q(q-1)^{n-1}$ proper colorings and every proper coloring is an odd coloring, we have $\chi_{odd}(T;q) \geq q(q-1)^{n-1}$. As $S_n$ has $n-1$ leaves, it also has at most $q(q-1)^{n-1}$ odd colorings, and so  $\chi_{odd}(S_n;q) = q(q-1)^{n-1}$.  

The remainder of the proof characterizes the trees that minimize the number of odd colorings. To do so, we characterize the trees $T$ admitting an odd coloring which is not proper. Suppose there is a path $P_4$ in $T$.  When $q>2$, we color the middle two vertices with color $1$ and the outer two vertices with color $2$. Then we color each neighbor of these four vertices with color $3$. From here, we iteratively complete the coloring by coloring an uncolored vertex $v$ that has a neighbor $w$ that is colored by assigning a color for $v$ that is distinct from the color on $w$. Note that the set of colored vertices is always connected and so each uncolored vertex has at most one colored neighbor. 
By construction, the vertices on the path $P_4$ all have color $2$ appearing once on their closed neighborhood, and any vertex $v$ not on the path $P_4$ has the color on $v$ appearing once on $N[v]$. Since this produces an odd coloring which is not proper, when $q>2$ the only tree that minimizes the number of odd colorings is the star.

Finally, assume that $q=2$. We show that there is an odd coloring of $T$ which is not proper if and only if $T$ has at least two vertices of even degree.  Suppose first that $T$ has two vertices of even degree, and let $v$ and $w$ denote a pair of distinct vertices with even degree that has the minimum positive distance between them.  This implies that each vertex on the path between $v$ and $w$ has odd degree.  We color the vertices of this path with color $1$, and then iteratively properly color the rest of $T$.  If $v'$ is not on the path, then the color on $v'$ is distinct from all neighbors of $v'$ and so appears once on $N[v']$. 
Since $v$ and $w$ have even degree, by parity considerations they have a color appearing an odd number of times on their closed neighborhoods.  Finally, any other vertices on the path between $v$ and $w$ have exactly three vertices in their closed neighborhood with color $1$.  This exhibits an odd coloring of $T$ which is not proper.

Now suppose that there is at most one even degree vertex in $T$, and suppose that there is a non-proper odd coloring of $T$. Consider a maximal length monochromatic path in a non-proper coloring of $T$.  One of the endpoints of this monochromatic path, say $v$, must have $d(v)$ odd, and by maximality $v$ has exactly two vertices receiving one color in $N[v]$.  But as the number of vertices in $N[v]$ is even it follows that $v$ has an even number of vertices colored with the other color, which contradicts the definition of odd coloring at vertex $v$.  Therefore there is no odd coloring which is not proper in a tree with at most one even degree vertex.
\end{proof}

\subsection{Conflict-free and non-monochromatic colorings --- Proofs of Theorems \ref{thm-conffree} and \ref{thm-nm}}

We now prove the extremal results for conflict-free and non-monochromatic colorings. Recall first that a conflict-free coloring is a vertex coloring where for every vertex $v$ there is a color occurring exactly once in $N[v]$.

\begin{proof}[Proof of Theorem \ref{thm-conffree}]
Let $T \in \calT(n)$. For minimizing, notice that every proper coloring of $T$ is a conflict-free coloring of $T$, and every conflict-free coloring of $T$ is an odd coloring of $T$. All odd colorings of $S_n$ are proper, and so all conflict-free colorings of the star are proper. Furthermore, the proof of Theorem \ref{thm-odd} shows that any tree with a $P_4$ has a non-proper conflict-free coloring when $q>2$.  

For $q=2$, we show that a tree $T$ has only proper conflict-free colorings if and only if $T$ does not contain a non-trivial subtree $T'$ with the property that each vertex in $T'$ has exactly one neighbor outside $T'$.
Suppose that a tree $T$ has a non-proper conflict-free coloring, and consider a color class containing a non-trivial component.  Notice that every vertex in that non-trivial component must have exactly one neighbor outside the component (which necessarily is colored with the other color).  Conversely, if a tree $T$ has a non-trivial subtree $T'$ in which each vertex in $T'$ has exactly one neighbor outside $T'$, then we can color $T'$ with color $1$ and iteratively properly color the rest of the tree to produce a non-proper conflict-free coloring of $T$.

We now turn to the maximization question.  First, notice that $\chi_{cf}(P_1;q) = q$, $\chi_{cf}(P_2;q) = q(q-1)$, and $\chi_{cf}(P_3;q) = q(q-1)^2$. For $P_n$ with $n \geq 4$ we denote the vertices as $v_1,v_2,\ldots,v_n$, and we enumerate the conflict-free colorings of $P_n$ by conditioning on whether or not $v_2$ and $v_3$ have the same color. If they don't, then deleting $v_1$ leaves a conflict-free coloring on the remaining path.  If they do, then deleting $v_1$ and $v_2$ leaves a conflict-free coloring of the remaining path, since in this case $v_3$ must have a different color from $v_4$.

Since in either case the only restriction for the color on $v_1$ is that it must differ from the color on $v_2$, we have the recurrence 
\[
\chi_{cf}(P_n;q) = (q-1)\chi_{cf}(P_{n-1};q) + (q-1)\chi_{cf}(P_{n-2};q) \qquad (n \geq 4).
\]
Similarly, we can condition on whether or not $v_{\ell+1}$ and $v_{\ell+2}$ have the same color, and from there properly color $v_{1},\ldots,v_{\ell}$ (giving $q-1$ choices for a color on these $\ell$ vertices). If $\ell \geq 2$, notice that there is a non-proper conflict-free extension 
in which $v_{2}$ and $v_{3}$ have the same color (from the case where $v_{\ell+1}$ and $v_{\ell+2}$ have distinct colors). This implies that 
\begin{equation}\label{eqn-pathcf}
    \chi_{cf}(P_n;q) > (q-1)^\ell \chi_{cf}(P_{n-\ell};q) + (q-1)^{\ell} \chi^{cf}(P_{n-\ell-1};q)  \qquad (\ell \geq 2).
\end{equation} 

With these calculations in hand, we move to proving that $\chi_{cf}(T;q) \leq \chi_{cf}(P_n;q)$. Note by the characterization of uniqueness for minimizing the number of conflict-free colorings we have $\chi_{cf}(S_n;q) < \chi_{cf}(P_n;q)$ for $n \geq 4$ and $q \geq 2$. We induct on $n$ to show that the path $P_n$ is the unique tree maximizing the number of conflict-free colorings. The base cases $n \leq 3$ are trivial, so suppose that $n \geq 4$, $q \geq 2$, and $T_n \neq S_n$ is a tree on $n$ vertices. 

Find a maximum length path in $T_n$ and let $v$ be a penultimate vertex on this path. Denote the leaves adjacent to $v$ by $u_1,u_2,\ldots,u_\ell$ and let $w$ be the non-pendant neighbor of $v$ (here we use that $T_n \neq S_n$).  Then partition the conflict-free colorings of $T_n$ based on whether $v$ and $w$ have the same color or not.

If they don't, we delete $u_1$, $u_2$,$\ldots$, $u_\ell$, which results in a conflict-free coloring of the remaining tree.  Since each $u_i$ must have a different color from the color on $v$, letting $T_{n-\ell}:=T_n - \{u_1,\ldots,u_\ell\}$ there are at most $(q-1)^\ell \chi_{cf}(T_{n-\ell};q)$ conflict-free colorings of $T_n$ where the colors on $v$ and $w$ are not the same. 

If the colors on $v$ and $w$ are the same, we delete $u_1$, $u_2,\ldots$, $u_\ell,$ and $v$, which results in a conflict-free coloring of the resulting tree $T_{n-\ell-1}:= T_n - \{v,u_1,\ldots,u_{\ell}\}$. In this case we know that $v$ and $w$ must have the same color, and each $u_i$ must have a color differing from $v$.  This produces an upper bound of $(q-1)^\ell \chi_{cf}(T_{n-\ell-1};q)$ for the number of conflict-free colorings of $T_n$ where the colors on $v$ and $w$ are the same.

Putting these together and using the inductive hypothesis along with (\ref{eqn-pathcf}), we have
\begin{eqnarray*}
\chi_{cf}(T_n;q) &\leq& (q-1)^\ell \chi_{cf}(T_{n-\ell};q) + (q-1)^\ell \chi_{cf}(T_{n-\ell-1};q)\\
&\leq& (q-1)^\ell \chi_{cf}(P_{n-\ell};q) + (q-1)^\ell \chi_{cf}(P_{n-\ell-1};q)\\
&\leq& \chi_{cf}(P_{n};q).
\end{eqnarray*}
Furthermore, the last inequality is an equality only when $\ell=1$ by (\ref{eqn-pathcf}). When $\ell=1$, we have strict inequality in moving from the first line to the second line unless deleting a leaf and deleting a leaf plus its neighbor from $T_n$ leaves $P_{n-1}$ and $P_{n-2}$, respectively, which implies that $T_n = P_n$.  
\end{proof}

The proof for non-monochromatic colorings is similar to conflict-free colorings. Recall that a non-monochromatic coloring is a vertex coloring where for every vertex $v$ there are at least two colors occurring in $N[v]$. We remark that every proper coloring is a non-monochromatic coloring.

\begin{proof}[Proof of Theorem \ref{thm-nm}]
We describe the changes needed to the proof of Theorem \ref{thm-conffree}.  Notice that again a leaf must receive a different color from its neighbor, and so the only non-monochromatic colorings of the star are proper colorings. For uniqueness, suppose that there exists a $P_4$ in a tree $T$.  Coloring the two leaves of the path with color $2$ and the remaining two vertices with color $1$, and then iteratively properly coloring the rest of the tree, produces a non-proper coloring of $T$.

For maximizing, in $P_n$ we again condition on whether or not $v_2$ and $v_3$ have the same color. If they don't, then we delete $v_1$; if they do, then we delete $v_1$ and $v_2$. This gives
\[
\chi_{nm}(P_n;q) = (q-1)\chi_{nm}(P_{n-1};q) + (q-1)\chi_{nm}(P_{n-2};q) \qquad (n \geq 4).
\]
We also have $\chi_{nm}(P_1;q) = q$, $\chi_{nm}(P_2;q) = q(q-1)$, and $\chi_{nm}(P_3;q) = q(q-1)^2$.

As before, for $\ell \geq 2$ we also can condition on $v_{\ell+1}$ and $v_{\ell+2}$ having the same color or not, giving
\[
\chi_{nm}(P_n;q) > (q-1)^\ell\chi_{nm}(P_{n-\ell};q) + (q-1)^\ell\chi_{nm}(P_{n-\ell-1};q) \qquad (\ell \geq 2).
\]

We prove the result for trees $T \neq S_n$ by induction on $n$, with the base cases trivial. With the above bounds in place, the induction proceeds exactly as in the proof of Theorem \ref{thm-conffree}.
\end{proof}

\subsection{Star rainbow colorings --- Proof of Theorem \ref{thm-rainbowcols}}

Here we prove the extremal results for star rainbow colorings. Recall that a star rainbow coloring is a vertex coloring where for every vertex $v$ all colors occur at most one time in $N[v]$. Notice that all star rainbow colorings are proper colorings and that every coloring that uses a distinct color for each vertex is a star rainbow coloring.

\begin{proof}[Proof of Theorem \ref{thm-rainbowcols}]
By considering the center of $S_n$, we see that each vertex in a star rainbow coloring of $S_n$ must have a different color.  Therefore we have $\chi_{sr}(S_n;q)=q(q-1)\cdots(q-n+1)$, which is the minimum number for any $n$-vertex tree.  Also, any tree containing a path $P_{4}$ admits a star rainbow coloring where the two ends of the $P_{4}$ have color $1$ and the remaining vertices have distinct colors.  This shows that $S_n$ uniquely minimizes the number of star rainbow colorings.

Notice that we can obtain the count of the number of star rainbow colorings of a tree $T$ by iteratively coloring $T$ starting with a leaf. 
To color a vertex $v$ (adjacent to a colored vertex $w$), the color on $v$ must avoid all of the distinct colors appearing on the neighbors of $N[w]$ that have already been colored.

Using $T=P_n$, we see that  $\chi_{sr}(P_n;q) = q(q-1)(q-2)^{n-2}$.  If $T_n \neq P_n$, then $T_n$ has a vertex of degree at least three, and so by iteratively coloring from a leaf there is one vertex with at most $q-3$ color possibilities. This gives  $\chi_{sr}(T_n;q) \leq q(q-1)(q-2)^{n-3}(q-3)$. Therefore the unique tree the the most number of star rainbow colorings is the path. 
\end{proof}

Notice that the iterative coloring procedure in the proof of Theorem \ref{thm-rainbowcols} gives the number of star rainbow colorings for any tree. In particular, 
let $T_n$ be a tree with $n$ vertices, and let $v_1$ be a leaf of $T_n$.  Let $v_2,v_3,...,v_n$ be iteratively chosen vertices so that the induced graph on $v_1,\ldots,v_i$ for each $i \in [2,n]$ forms a tree.  Then the number of star rainbow colorings of $T_n$ is given by 
\[
\chi_{sr}(T_n;q) = q\cdot \prod_{i=2}^n \left(q - |\{k : k<i, j<i, v_i \sim v_j \sim v_k\}|-1\right).
\]
In words, when coloring vertex $v_i$, we must avoid the color on its unique neighbor $v_j$ that already has a color, and also any colors appearing on a vertex $v_k$ that is in the closed neighborhood of $v_j$.

\subsection{2-Strong-Conflict-Free Colorings --- Proof of Theorem \ref{thm-2scf}}

Here we prove the lower bound for 2-strong-conflict-free colorings of a tree. We first recall that a 2-strong-conflict-free coloring is a vertex coloring where for every vertex $v$ there are at least two colors that occur exactly once in $N[v]$.

\begin{proof}[Proof of Theorem \ref{thm-2scf}]
Note that $\chi_{2scf}(P_n;q) = q(q-1)(q-2)^{n-2}$, since the 2-strong-conflict-free colorings of $P_n$ are exactly the star rainbow colorings of $P_n$. 

Now let $T$ be an $n$-vertex tree, and root $T$ at a leaf.  Color the root and its unique neighbor, and then iteratively color out so that each vertex receives a different color than the color on the two closest vertices on the path to the root.  This produces a 2-strong-conflict-free coloring of $T$ and so $\chi_{2scf}(T;q) \geq q(q-1)(q-2)^{n-2}$. If $v$ is a vertex with degree at least three, then one of the two neighbors of $v$ that is not closest to the root can receive the same color as the neighbor of $v$ closest to the root. Therefore if $T \neq P_n$ then $\chi_{2scf}(T;q)>q(q-1)(q-2)^{n-2}=\chi_{2scf}(P_n;q)$.
\end{proof}

\subsection{Star colorings --- Proof of Theorem \ref{thm-starcols}}

Here we prove the extremal result for star colorings. Recall that a star coloring is a vertex coloring that is proper and has no 2-colored path $P_4$. 

\begin{proof}[Proof of Theorem \ref{thm-starcols}]
Star colorings are proper colorings with color restrictions on all $P_4$ subgraphs.  Since $S_n$ is the unique tree that has no $P_4$ subgraph, all proper colorings of $S_n$ are star colorings, and every other tree admits a proper coloring that is not a star coloring.  In particular, $S_n$ uniquely maximizes the number of star colorings. 

We now turn to minimizing the number of star colorings.  As in the proof of Theorem \ref{thm-conffree}, we first find a recursion for the number of star colorings of paths, and to start we have $\chi_s(P_1;q)=q$, $\chi_s(P_2;q)=q(q-1)$, and $\chi_s(P_3;q)=q(q-1)^2$.  For $n\geq 4$ suppose the path $P_n$ has vertices $v_1,v_2,\ldots,v_n$, and we partition the star colorings of $P_n$ based on whether $v_1$ and $v_3$ have the same color or not.  If they do, then $v_2$ and $v_4$ must have different colors, and those colors must also differ from the color on $v_3$. Deleting $v_1$ and $v_2$ produces a star coloring of $P_{n-2}$, and furthermore any star coloring of this $P_{n-2}$ extends to $q-2$ star colorings of $P_n$ of this type. 

If $v_1$ and $v_3$ have different colors, then we delete $v_1$ and are left with a star coloring of $P_{n-1}$.  But any star coloring of $P_{n-1}$ extends to $q-2$ star colorings of $P_n$ of this type, as $v_2$ and $v_3$ have different colors so we can color $v_1$ with any color that differs from those on $v_2$ and $v_3$. 

Putting this together, we have
\[
\chi_{s}(P_n;q) = (q-2)\chi_{s}(P_{n-1};q) + (q-2)\chi_{s}(P_{n-2};q) \qquad (n \geq 4).
\]

We now show that $\chi_{s}(P_n;q) \leq \chi_{s}(T;q)$. 
We induct on $n$ to show that the path $P_n$ is the unique tree minimizing the number of star colorings. The base cases $n \leq 3$ are trivial, so suppose that $n \geq 4$, $q \geq 2$, and $T_n$ is a tree on $n$ vertices. By the uniqueness of the upper bound we have $\chi_s(S_n;q) > \chi_{s}(P_n;q)$, so we may assume $T_n \neq S_n$.
Let $v$ be a vertex with at most one non-pendant neighbor (take, for example, $v$ to be a penultimate vertex in a maximal path). We let distinct vertices $u$, $w$, and $x$ be such that $u$ is a leaf and $u \sim v \sim w \sim x$, and partition the star colorings of $T_n$ based on the colors on $u$ and $w$. 

If the colors on $u$ and $w$ are different, we delete $u$. This gives a star coloring of $T_{n-1} := T_n - u$. But all star colorings of $T_n$ of this type come from a star coloring of $T_{n-1}$ by giving $u$ a different color from $v$ and $w$, so there are exactly $(q-2)\chi_s(T_{n-1};q)$ colorings of $T_n$ where the colors on $u$ and $w$ are distinct.

If the colors on $u$ and $w$ are the same, we delete $v$ and the $\ell$ pendant neighbors of $v$ (including $u$).  We recover all star colorings of $T_n$ of this type from a star coloring of the remaining tree $T_{n-\ell-1}$ by giving $v$ a color different from the colors on $w$ and $x$, $u$ the color on $w$, and all other pendants a color that differs from $v$. Therefore there are exactly $(q-2)(q-1)^{\ell-1}\chi_s(T_{n-\ell-1};q)$ colorings of $T_n$ where the colors on $u$ and $w$ are the same.

The preceding arguments show
\[
\chi_s(T_n;q) = (q-2)\chi_s(T_{n-1};q) + (q-2)(q-1)^{\ell-1}\chi_s(T_{n-\ell-1};q).
\]
Now any star coloring of $P_{n-2}$ (with vertices $v_1,\ldots,v_{n-2}$) is obtained by starting with a star coloring of $P_{n-\ell-1}$ (vertices $v_1,\ldots,v_{n-\ell-1}$), and iteratively coloring the remaining vertices as a star coloring. Since a star coloring is a proper coloring, this implies that there are at most $q-1$ choices for a color on each of $v_{n-\ell},\ldots,v_{n-2}$, which gives $(q-1)^{\ell-1} \chi_s(P_{n-\ell-1};q) \geq \chi_s(P_{n-2};q)$.

Therefore, we see that
\begin{eqnarray*}
\chi_s(T_n;q) &=& (q-2)\chi_s(T_{n-1};q) + (q-2)(q-1)^{\ell-1}\chi_s(T_{n-\ell-1};q)\\
&\geq& (q-2)\chi_s(T_{n-1};q)+(q-2)(q-1)^{\ell-1}\chi_s(P_{n-\ell-1};q)\\
&\geq& (q-2)\chi_{s}(P_{n-1};q) + (q-2)\chi_{s}(P_{n-2};q)\\
&=&\chi_s(P_n;q).
\end{eqnarray*}
This proves the inequality, and so we move now to the characterization of equality. We first argue that to have equality requires $\ell=1$, which will follow from showing that $(q-1)^{\ell-1}\chi_s(P_{n-\ell-1};q)>\chi_s(P_{n-2};q)$ for $\ell>1$.  
Given a star coloring of $v_1,\ldots,v_{n-\ell-1}$, we use an iterative proper coloring of $v_{n-\ell},\ldots,v_{n-2}$ for the non-strict inequality.  But if $\ell>1$, one such coloring will have the same colors on $v_{n-\ell-1}$ and $v_{n-\ell-3}$ and will choose a color on $v_{n-\ell}$ that appears on $v_{n-\ell-2}$.
This creates a 2-colored $P_4$ in $P_{n-2}$, and so is a proper coloring extension that is not a star coloring of $P_{n-2}$.  Therefore $(q-1)^{\ell-1}\chi_s(P_{n-\ell-1};q)>\chi_s(P_{n-2};q)$ for $\ell>1$.  

Finally, if $\ell=1$, then by induction we have equality only when $T_{n-1}$ and $T_{n-2}$ are $P_{n-1}$ and $P_{n-2}$. respectively, which implies that $T_n=P_n$.
\end{proof}

\subsection{Existence Homomorphisms --- Proof of Theorem \ref{thm-xhom}}

Here we prove the results about existence homomorphisms to $E_q^\circ$. Recall that an existence homomorphism from $T$ to $E_{q}^\circ$ is a vertex coloring of $T$ so that no color class contains an isolated vertex. 

\begin{proof}[Proof of Theorem \ref{thm-xhom}]
For minimizing, notice that each tree can be monochromatically colored.  For $S_n$, these are the only possible colorings, since a leaf must have the same color as its neighbor.  And if a tree $T$ has a path $P_4$, then coloring one leaf and its neighbor $v$ with color $1$ and the other leaf and its neighbor $w$ with color $2$, and monochromatically coloring the vertices in each component of $T-vw$, we see that every $T \neq S_n$ has an existence homomorphism to $E_q^\circ$ (with $q \geq 2$) that is not a monochromatic coloring of $T$.

For maximizing, we again begin by giving a recursive definition for $\xhom(P_n,E_q^\circ)$.  As a leaf must have the same color as its neighbor, we have that $\xhom(P_2,E_q^\circ) = q$ and $\xhom(P_3,E_q^\circ) = q$. Let $P_n$ have vertices $v_1,v_2,\ldots,v_n$, and consider an existence homomorphism from $P_n$ to $E_{q}^\circ$. 
If the colors on $v_2$ and $v_3$ are the same, then this color must also be the color on $v_1$, and we delete $v_1$ and obtain an existence homomorphism from the remaining path to $E_{q}^\circ$. If the colors on $v_2$ and $v_3$ differ, then we delete $v_1$ and $v_2$ to obtain an existence homomorphism from the remaining path to $E_q^\circ$. Since in this latter case the color on $v_1$ and $v_2$ has $(q-1)$ possibilities, this gives the recurrence
\[
\xhom(P_n,E_q^\circ) = \xhom(P_{n-1},E_q^\circ) +(q-1)\xhom(P_{n-2},E_q^\circ) \qquad (n \geq 4).
\]

Given $\ell$ with $2 \leq \ell \leq n-3$ we can condition on whether or not $v_{\ell+1}$ and $v_{\ell+2}$ have the same color. By monochromatically coloring $v_1,\ldots,v_\ell$, we see that
\begin{equation}\label{eq-xhom}
\xhom(P_n,E_q^\circ) > \xhom(P_{n-\ell},E_q^\circ)+(q-1)\xhom(P_{n-\ell-1},E_q^\circ),
\end{equation}
with strict inequality coming from a coloring of $P_n$ in which $v_2$ and $v_3$ have distinct colors (in the case where $v_{\ell+1}$ and $v_{\ell+2}$ have the same color).

Now we induct on $n$ to show that $P_n$ is the unique $n$-vertex tree $T_n$ maximizing $\xhom(T_n,E_q^\circ)$. Notice that by the minimizing result, we may assume that $T_n \neq S_n$ as $\xhom(S_n,E_q^\circ) = q < \xhom(P_n,E_q^\circ)$ for $n \geq 4$.
Consider a maximum length path in $T_n$ and let $v$ be a penultimate vertex on this path. Denote the leaves adjacent to $v$ by $u_1,\ldots,u_\ell$ (so that $\ell \leq n-3$) and let $w$ be a non-pendant neighbor of $v$.  Then we enumerate the existence homomorphisms to $E_q^\circ$ based on whether or not $v$ and $w$ have the same color. Note that the color on $u_1,\ldots,u_{\ell}$ must be the same as the color on $v$.

If $v$ and $w$ have the same color, delete $u_1,\ldots,u_\ell$ and obtain a tree $T_{n-\ell}$ that has an existence homomorphism to $E_{q}^\circ$. Each existence homomorphism of $T_{n-\ell}$ comes from one existence homomorphism of $T_{n}$ of this type and vice versa, and so there are  $\xhom(T_{n-\ell},E_{q}^{\circ})$ existence homomorphisms of $T_n$ so that $v$ and $w$ have the same color.

When $v$ and $w$ do not have the same color, we delete  $u_1,\ldots,u_\ell$ and $v$. This produces a tree $T_{n-\ell-1}$ that has an existence homomorphism to $E_q^\circ$. Since the color of the deleted vertices must differ from the color on $w$,  
there are  $(q-1)\xhom(T_{n-\ell-1},E_q^\circ)$ existence homomorphisms from $T_n$ to $E_q^\circ$ where $v$ and $w$ do not have the same color.

By the remarks above, induction, and (\ref{eq-xhom})  we have
\begin{eqnarray*}
\xhom(T_n,E_q^\circ) &=& \xhom(T_{n-\ell},E_q^\circ) + (q-1)\xhom(T_{n-\ell},E_q^\circ)\\
&\leq& \xhom(P_{n-\ell},E_q^\circ) + (q-1)\xhom(P_{n-\ell-1},E_q^\circ)\\
&\leq& \xhom(P_n,E_q^\circ).
\end{eqnarray*}
For equality, by (\ref{eq-xhom}) we have $\ell=1$.  Given this, by induction we have $T_{n-1}=P_{n-1}$ and $T_{n-2} = P_{n-2}$, which implies that $T_n=P_n$.
\end{proof}



\begin{thebibliography}{99}

\bibitem{Abel}
M. Abellanas, P. Bose, J. Garcia, F. Hurtado, M. Nicolas, and P. A. Ramos, On properties of higher order delaunay graphs with applications,  in: {\em Proc. 21st European Workshop on Computational Geometry}, (2005), 9--11.

\bibitem{CKP}
P. Cheilaris, B. Keszegh, D. P\'alv\"olgyi,  Unique-maximum and conflict-free coloring for hypergraphs and tree graphs, {\em SIAM J. Discrete Math.} {\bf 27(4)} (2013), 1775--1787.

\bibitem{CheilarisToth}
P. Cheilaris and G. T\'oth, Graph unique-maximum and conflict-free colorings, {\em J. Discrete Algorithms} {\bf 9(3)} (2011), 241--251.

\bibitem{Cutler}
J. Cutler, Coloring graphs with graphs: a survey, {\em Graph Theory Notes N.Y.} {\bf 63} (2012), 7-16.  

\bibitem{CutlerRadcliffe}
J. Cutler and A.J. Radcliffe, Counting dominating sets and related structures in graphs, {\em Discrete Math.} {\bf 339} (2016), 1593–-1599.

\bibitem{CsikvariLin}
P. Csikv\'{a}ri and Z. Lin, Graph homomorphisms between trees, {\em Electronic J. Combinatorics} {\bf 21(4)} (2014), \#P4.9. 

\bibitem{EngbersGalvin}
J. Engbers and D. Galvin, Extremal $H$-colorings of trees and 2-connected graphs, {\em J. Comb. Theory Ser. B} {\bf 122} (2017), 800--814.


\bibitem{NP}
G. Even, Z. Lotker, D. Ron, and S. Smorodinsky, Conflict-free colorings of simple geometric regions with applications to frequency assignment in cellular networks, {\em SIAM J. Comput.} {\bf 33} (2003), 94--136.

\bibitem{FRR}
G. Fertin, A. Raspaud, and B. Reed, Star Coloring of Graphs, {\em J. Graph Theory} {\bf 47(3)} (2004), 163--182.


\bibitem{Grunbaum}
B. Gr\"unbaum, Acyclic colorings of planar graphs, {\em Israel J. Math.} {\bf 14(3)} (1973), 390--408.

\bibitem{ProdingerTichy}
H. Prodinger and R. Tichy, Fibonacci numbers of graphs, {\em The Fibonacci Quart.} {\bf 20} (1982), 16--21.

\bibitem{Sidorenko}
A. Sidorenko, A partially ordered set of functionals corresponding to graphs, {\em Discrete Math.} {\bf 131} (1994), 263--277.

\bibitem{SmorThesis}
S. Smorodinsky, Combinatorial Problems in Computational Geometry, PhD thesis,
School of Computer Science, Tel-Aviv University, 2003.

\bibitem{Smor}
S. Smorodinsky, Conflict-free coloring and its applications,  in: I. B\'ar\'any et al. (eds.), Geometry
-- Intuitive, Discrete, and Convex, Bolyai Soc. Math. Stud., Springer, Vol 24, 2013, 331--389.

\bibitem{Zhao}
Y. Zhao, Extremal regular graphs: independent sets and graph homomorphisms, 
{\em Amer. Math. Monthly} {\bf 124} (2017), 827--843. 


\end{thebibliography}
\end{document}